\documentclass[reqno, a4paper,12pt]{amsart}

\usepackage[capposition=top]{floatrow}
\usepackage{euscript,eufrak,verbatim}
\usepackage{graphicx}
\usepackage[usenames]{color}
\usepackage[colorlinks,linkcolor=red,anchorcolor=blue,citecolor=blue]{hyperref}
\usepackage{amsmath, mathtools}
\usepackage{mathtools}
\usepackage{amsthm}
\usepackage[all]{xy}
\usepackage{amssymb}
\usepackage{bm}

\usepackage[all]{xy}

\usepackage{mathrsfs}
\usepackage{amscd}

\usepackage{enumitem}

\makeatletter
\numberwithin{equation}{section}

\theoremstyle{plain}
\newtheorem{main theorem}{Main Theorem}
\newtheorem{theorem}{Theorem}[section]
\newtheorem{lemma}[theorem]{Lemma}

\newtheorem{corollary}[theorem]{Corollary}
\newtheorem{proposition}[theorem]{Proposition}

\theoremstyle{definition}
\newtheorem{definition}[theorem]{Definition}
\newtheorem{remark}[theorem]{Remark}

\newtheorem{question}[theorem]{Question}
\begin{document}
\title{Minimal amenable subshift with full mean dimension}

\author[Zhengyu Yin, Zubiao Xiao]{Zhengyu Yin, Zubiao Xiao*}

\subjclass{ 37B99, 54H15}

\keywords{Dynamical systems,  Amenable group, Tiling, Minimal subshift, Mean topological diemnsion}

\maketitle

\begin{abstract}
Let $G$ be an infinite countable amenable group and $P$ a polyhedron with topological dimension $dim(P)<\infty$. We construct a minimal subshift $(X,G)$ such that its mean topological dimension is equal to $dim(P)$. This result answers the question of D. Dou in \cite{DD}, moreover, it is also an extension of the work of L. Jin and Y. Qiao \cite{JQ} for $\mathbb{Z}$-action.
\end{abstract}

\section{Introduction} \label{section: introduction}
In 1999, M. Gromov in \cite{MG} introduced the notion of mean dimension for a topological dynamical system $(X,f)$, where $X$ is a compact metric space and $f$ a continuous self-map on $X$, which is, as well as the topological entropy, an invariant under conjugacy.
In 2000, E. Lindenstrauss and B. Weiss in \cite{BL} showed that a system $(X,f)$ with finite dimension phase space has zero dimension, they also proved that the values of mean topological dimension of $\mathbb{Z}$-action can run over $[0,\infty]$. Due to this result, they also negatively answered the question asked by J. Auslander in \cite{Au}: Whether all minimal $\mathbb{Z}$-action systems can be embedded into $\left([0,1]^{\mathbb{Z}},\mathbb{Z}\right)$.  In \cite{el}, E. Lindenstrauss proved that a system $(X,f)$ which is an extension of a minimal system, and $K$ is a convex set with a non-empty interior such that $mdim(X,f)<dim(K)/36$, then $(X,f)$ can be embedded in the shift $\left(K^\mathbb{Z},\mathbb{Z}\right)$. Later, E. Lindenstrauss and M. Tsukamoto in \cite{LT} constructed a minimal system with mean dimension equal to $m/2$ that can't embed into $\left(([0,1]^m)^\mathbb{Z},\mathbb{Z}\right)$. More recently, Y. Gutman and M. Tsukamoto in \cite{GT} showed
that if $(X,f)$ is a minimal system with $mdim(X,f)<m/2$, then $(X,f)$ can be embedded into $\big(([0,1]^m)^\mathbb{Z},\mathbb{Z}\big)$. Therefore, $m/2$ is the optimal number for minimal systems embedding into $\big(([0,1]^m)^\mathbb{Z},\mathbb{Z}\big)$
.

Given an infinite countable amenable group $G$, It is natural to discuss the value of the mean dimension of the system $(X,G)$. Let $P$ be a polyhedron with topological dimension $dim(P)$ and  $(P^G, G)$ the $G$-shift on $P$. Given $0\leq\rho\leq dim(P)$ and $G$ a residual finite amenable group, M. Coornaert and F. Krieger in \cite{CK} constructed a closed subshift $X$ of $P^G$ (need not to be minimal) such that $mdim(X,G)=\rho$. Soon after, let $G$ be an infinite countable amenable group (without any additional assumption) and $\epsilon>0$. F. Krieger in \cite{FK} constructed a minimal subshift $X$ of $P^G$ such that its mean topological dimension is $\epsilon$-close to $dim(P)$. More recently, by applying the methods of construction of Toeplitz sequence in \cite{SW} and the newest results of the tiling theory of amenable groups in \cite{DHZ}, let $0\leq\rho< dim(P)$, D. Dou \cite{DD} constructed a minimal subshift $X$ of $P^G$ with $mdim(X, G)=\rho$, and a question was left:

\begin{question}\rm{(\cite[Question 1]{DD})}\label{que1}
Let $G$ be an infinite countable amenable group and $P$ a polyhedron with topological dimension $dim(P)$, whether it is possible to construct a minimal $G$-action subshift $X$
of $P^G$ with full mean topological dimension (i.e., $mdim(X, G)=dim(P)$).
\end{question}

In the case of $\mathbb{Z}$-action, L. Jin and Y. Qiao \cite{JQ} constructed a minimal $\mathbb{Z}$-subshift $X$ of $[0,1]^{\mathbb{Z}}$ with $mdim(X,\mathbb{Z})=1$, which positively answers the Question \ref{que1} for $\mathbb{Z}$-action. Soon after, J. Zhao \cite{ZJ} considered the construction of a minimal system of given mean dimension in \textit{Berstein space}, which generalizes the results in \cite{DD}. In this paper, we extend L. Jin and Y. Qiao's result in \cite{JQ} to infinite countable amenable group $G$ action and give a full answer to Question \ref{que1}. More precisely, we obtain the following result:
\begin{theorem}\rm\label{mt}
    Let $G$ be an infinite countable amenable group and $P$ a polyhedron with topological dimension $dim(P)$, there is a minimal subshift $X$ of $P^G$ such that $mdim(X,G)=dim(P)$.
\end{theorem}

This paper is organized as follows. In section 2, we recall some basic concepts of topological dynamical systems and amenable groups via the tiling technique. Besides, the definition and some fundamental lemmas of mean topological dimension are also included. In section 3, we prove the Theorem \ref{mt}.
\section{Preliminaries}

\subsection{Amenable groups}

Let $G$ be an infinite countable group and denote $\mathcal{F}(G)$ the class of all finite non-empty subsets of $G$. Recall that a group $G$ is \textit{amenable} if it admits a sequence of non-empty finite subsets $\{F_n:F_n\in \mathcal{F}(G), n\in \mathbb{N}\}$ such that for any  $K\in \mathcal{F}(G)$, we have
\begin{center}
    $\underset{n\to \infty}{\lim}\,\displaystyle\frac{|B(F_n,K)|}{|F_n|}=0$,
\end{center}
where $B(F_n,K)=\{g\in G:Kg\cap F_n\ne \emptyset$ and $Kg\cap (G\backslash F_n)\ne\emptyset\}$ and $|\cdot|$ denotes the cardinality operator. Moreover, we call such sequence $\{F_n:F_n\in \mathcal{F}(G), n\in\mathbb{N}\}$ a \textit{F$\phi$lner sequence} of $G$. Throughout the paper, we assume that $G$ is an infinite countable amenable group.

\subsection{Dynamical systems and shifts}
We call $(X,G)$ a \textit{$G$-space} if $G$ acts on compact metric space $X$  such that $e_Gx=x$, $g_1(g_2x)=(g_1g_2)x$ for any $g_1,g_2\in G$ and $x\in X$, where $e_G$ is the identity of $G$. Let $(X, G)$ be a $G$-space with metric $d$ on $X$, given $F\in \mathcal{F}(G)$, we define 
\begin{equation*}
    d_F(x,y)=\max_{g\in F}\,d(gx,gy).
\end{equation*}
Clearly, $d_F$ is a compatible metric on $X$ for any $F\in \mathcal{F}(G)$.

 A subset $A\subset G$ is said to be \textit{syndetic} if there is a $K\in \mathcal{F}(G)$ such that $KA=G$, where $KA=\{ka:k\in K, a\in A\}$.
Let $(X, G)$ be a $G$-space, $x\in X$ and $U$ a neighborhood of $x$, we denote $N_G(x,U)=\{g\in G:gx\in U\}$ by the \textit{return-time} set of $x$ with respect to $U$.
We say $x\in X$ a \textit{almost periodic point} if for any neighborhood $U$ of $x$, there is a syndetic subset $A_U\subset G$ such that $A_U\subset N_G(x,U)$. For any $x\in X$, denote $\mathcal{O}(x)=\{gx:g\in G\}$ the orbit of $x.$
and $\overline{\mathcal{O}x}$ the orbit-closure of $x$. 

The $G$-space $(X,G)$ is said to be \textit{minimal} if $\overline{Gx}=X$ for all $x\in X$. It is known that $(X,G)$ is minimal if $X$ is an orbit-closure of some almost periodic point $x\in X$ \big(see \cite[Theorem 1.7]{Au}\big).

Let $P$ be a compact metric space and $P^G$ the product space endowed with the product topology. The $G$-\textit{shift} defined on $P^G$ is given by 
\[g'(x_g:g\in G)=(x_{gg'}:g\in G),\]
then $(P^G,G)$ is a $G$-space. If $X$ is a $G$-invariant closed subset of $P^G$, we call $(X,G)$ a $G$-\textit{subshift}.

\subsection{Tiling of amenable group}
The definition of tiling of a discrete countable group is originally from \cite{OW}. In \cite{DHZ}, the authors proved that for any infinite countable amenable group $G$, given $K\in \mathcal{F}(G)$ and $\epsilon>0$, there is a tiling with finite shapes such that all the tiles are $(K,\epsilon)$-invariant. In \cite{DD},
The author used the tiling technique to construct a minimal system with a specific mean dimension. In this paper, to construct a minimal subshift with full mean topological dimension the tiling technique of an amenable group is also required.
\begin{definition}\rm{(\cite{OW,DHZ})}
    We say that $\mathcal{T}\subset \mathcal{F}(G)$ is a \textit{tiling} of $G$ if $\mathcal{T}$ is a disjoint union of $G$, that is, for any $T,T'\in \mathcal{T}$, $T\cap T'=\emptyset$ and $\cup_{T\in \mathcal{T}}T=G$.
Each $T\in \mathcal{T}$ is called a \textit{tile}.
    \end{definition}

  \begin{definition}\rm{(\cite{DHZ})}
          A tiling $\mathcal{T}$ is said to be \textit{finite}, if there is a finite collection $\mathcal{S}=\{S_1,...,S_n\}$ of elements of $\mathcal{F}(G)$ such that each tile $T\in \mathcal{T}$ is a right transformation of some $S_i\in \mathcal{S}$, i.e., there is some $c\in G$ such that $T=S_ic$. We call $\mathcal{S}$ a \textit{shape} of $\mathcal{T}$.
        For each $S_i\in \mathcal{S}$, we denote
        \begin{equation*}
            C(S_i)=\{c:S_ic\in \mathcal{T}\},
        \end{equation*}
        by the \textit{center} of $S_i$.
  \end{definition}
    For convenience, we assume that each element in $\mathcal{S}$ can not be the right translation of others and contains the identity $e_G$.

       \begin{definition}\rm{(\cite{DD})}
            A finite tiling  $\mathcal{T}$  with shape $\mathcal{S}$ is said to be \textit{irreducible} if $C(S_i)$ is syndetic for each $S_i\in \mathcal{S}$, i.e., for any $S_i\in \mathcal{S}$ there is a $K_i\in\mathcal{F}(G)$ such that $K_iC(S_i)=G$.
       \end{definition}

To construct the minimal subshift, we need to borrow some useful results which have been proved in like \cite{DHZ,DD} and also been restated in \cite{JPQ}.

 \begin{lemma}\rm{(\cite[Lemma 3.2]{DD})}\label{prop4}
Suppose $\mathcal{T}$ is a finite irreducible tiling of $G$ with shape $\mathcal{S}$, then for any $S_i\in\mathcal{S}$ and $n\in \mathbb{N}$, there is a $T\in\mathcal{F}(G)$ and $\epsilon>0$ such that whenever $K\in\mathcal{F}(G)$ is $(T,\epsilon)$-invariant, $K$ contains at least $n$ elements in $\mathcal{T}$ with shape $S_i$.
 \end{lemma}

\begin{definition}\rm{(\cite{DHZ})}
Now, let $\{\mathcal{T}_n\}_{n\in \mathbb{N}}$ be a sequence of finite tiling, $\{\mathcal{T}_n\}_{n\in \mathbb{N}}$ is said to be \textit{primely congruent} if for each $n\in \mathbb{N}$, $\mathcal{T}_n$ is uniquely refined by $\mathcal{T}_{n-1}$, namely, every $T_{n}\in\mathcal{T}_n$
 is a finite union of some elements of $\mathcal{T}_{n-1}$ and each shape of $\mathcal{T}_n$ is partitioned by shapes of $\mathcal{T}_{n-1}$ in a unique way \big(i.e., for any two $\mathcal{T}_n$-tiles $Sc_1$ and $Sc_2$ of the same shape $S\in \mathcal{S}_n$, we have $\mathcal{T}_n|_{Sc_1}=(\mathcal{T}_n|_{Sc_2})c_2^{-1}c_1\big)$
 \end{definition}
\begin{lemma}\rm{(\cite[Theorem 5.2]{DHZ}, \cite[Theorem 3.6]{DD})}\label{teo4}
    Let $G$ be an infinite countable amenable group with the identity $e_G$, $\{K_n\}_{n\in \mathbb{N}}$ a sequence of finite subsets of $G$ with ${\bigcup}_{n=1}^{\infty}K_n=G$, and $\{\epsilon_n\}_{n\in \mathbb{N}}$ a decreasing sequence of positive numbers converging to zero. Then there is a primely congruent sequence $\{\mathcal{T}_n\}_{n\in \mathbb{N}}$ of finite irreducible tilings satisfying the following conditions:
    \begin{enumerate}
        \item [\rm{(1)}] $e_G\subset S_{1,1}\subset S_{2,1}\subset\cdots\subset S_{n,1}\subset \cdots\subset \bigcup_{n=1}^{\infty}S_{n,1}=G$;

        \item [\rm{(2)}] for each $n\in \mathbb{N}$ and every $S_{n,i}\in \mathcal{S}_n$, $S_{n,i}$ is $(K_n,\epsilon_n)$-invariant, for all $1\leq i\leq m_n$,
    \end{enumerate}
    where $\mathcal{S}_n=\{S_{n,i}:1\leq i\leq m_n\}$ is the shape of $\mathcal{T}_n$.
\end{lemma}

\subsection{Topological dimension and Mean topological dimension}
In this subsection, we recall the notations of the topological dimension and the mean topological dimension of $(X, G)$.

Let $(X,d)$ be a compact metrizable space and $\alpha=\{U_1,...,U_n\}$, $\beta=\{V_1,...,V_m\}$ finite open covers of $X$, $\beta$ is said to be \textit{finer} than $\alpha$ (written as $\alpha\preceq\beta$) if for any $V_i\in \beta$ there is a $U_j\in \alpha$ such that $V_i\subset U_j$. 

The \textit{order} of $\alpha$ is defined by
\begin{equation*}
    ord(\alpha)=\left(\max_{x\in X}\sum_{i=1}^n 1_{U_i}(x)\right)-1,
\end{equation*}
where $1_{U_i}$ is the character function of $U_i$, and we define a quantity $D(\alpha)$ by
\begin{equation*}
    D(\alpha)=\min_{\alpha\preceq\beta}ord(\beta),
\end{equation*}
where $\beta$ runs over all finite open covers of $X$ that are finer than $\alpha$.

\begin{definition}\rm{(\cite{HW})}
   Let $X$ be a compact topological space, the \textit{topological dimension} of $X$ is given
\begin{equation*}
    dim(X)=\sup_{\alpha}D(\alpha),
\end{equation*}
where $\alpha$ runs over all finite open cover of $X$.
\end{definition}

    Let $(X,d)$ be a compact metric space,
 $P$ a compact metric space, and $\epsilon>0$ a positive number. A continuous map $f:X\to P$ is said to be \textbf{$\epsilon$-injective} if $d(x_1,x_2)\leq \epsilon$ for all $x_1,x_2\in X$ such that $f(x_1)=f(x_2)$.
Let $\epsilon>0$ be a real number, we define
    \begin{equation*}
        dim_\epsilon(X,d)=\inf_Pdim(P),
    \end{equation*}
    where $P$ runs overall compact metric spaces for which there exists an $\epsilon$-injective continuous map $f$ from $X$ to $P$.

    \begin{remark}(\cite[Proposition 4.6.2]{C1})
        Let $(X,d)$ be a compact metric space. It is known that
    \begin{equation*}
        \lim_{\epsilon\to 0}dim_\epsilon(X,d)=dim(X).
    \end{equation*}
    \end{remark}

To prove the main theorem, We need some fundamental Lemmas:
    \begin{lemma}\label{prop6}
    Let $(X,d)$ and $(X',d')$ be two compact metric spaces. Suppose there is a continuous map $f:X\to X'$ and a positive number $c>0$ satisfying
        $d(x,y)\leq c\cdot d'(f(x),f(y))$
     for all $x,y\in X$. Then for any $\epsilon>0$
     \begin{equation*}
         dim_{c\epsilon}(X,d)\leq dim_\epsilon(X',d').
     \end{equation*}

\end{lemma}
\begin{proof}
    The result follows that if $\varphi:X'\to P$ is an $\epsilon$-injective map, then $\varphi\circ f:X\to P$ is $c\epsilon$-injective.
\end{proof}

\begin{lemma}\rm{(\cite[Corollary 2.8]{CK})}\label{prop7}
    Let $(P,d)$ be a polyhedron. For each $n\in \mathbb{N}$, put $d_n(x,y)=\max_{1\leq i\leq n}d(x_i,y_i)$, $x=(x_i)\in P^n, y=(y_i)\in P^n$. There is an $\epsilon_0>0$ which does not depend on $n$ such that for each $\epsilon<\epsilon_0$, we have
\begin{equation*}
dim_\epsilon(P^n,d_n)=n\cdot dim(P).
\end{equation*}
\end{lemma}

      Let $(X,G)$ be a topological dynamical system with metric $d$ on $X$, $\epsilon>0$ and $\{F_n\}$ a F$\phi$lner sequence of $G$. The quantity $mdim_\epsilon(X,G)$ is defined by
    \begin{equation*}
        mdim_\epsilon(X,d,G)=\lim_{n\to \infty}\displaystyle\frac{dim_\epsilon(X,d_{F_n})}{|F_n|},
    \end{equation*}
    where $d_{F_n}$ is the compatible metric on $X$ with respect to $F_n$.
    
    Moreover, the value $mdim_\epsilon(X,d,G)$ is independent of the choice of F$\phi$lner sequence.
    \begin{definition}
          The \textit{mean topological dimension} of $(X,G)$ is defined as
    \begin{equation*}
        mdim(X,G)=\lim_{\epsilon\to 0}mdim_\epsilon(X,d,G).
    \end{equation*}
    \end{definition}
    \begin{remark}
        In this article, we use the equivalent definition of the mean topological dimension defined by metric approaches, one can see \cite{C1, CK, MG} for more details.
    \end{remark}

The next proposition is used in the proof of Theorem \ref{mt} which is originally due to E. Lindenstrauss and B. Weiss in \cite{BL} for $\mathbb{Z}$-actions to estimate mean topological dimension from below.

Recall that the \textit{upper density} of a subset $J$ of $G$ is defined by
\begin{equation*}
    \delta(J)=\sup_{\{F_n\}}\limsup_{n\to \infty}\displaystyle\frac{|J\cap F_n|}{|F_n|},
\end{equation*}
where $\{F_n\}$ runs over all F$\phi$lner sequences in $G$.

\begin{proposition}\label{prop10}
    Let $P$ be a polyhedron with metric $d$ and $X$ a subshift of $P^G$. Suppose there is a subset $J$ of $G$ and a point $x_0\in X$ satisfying
    \begin{enumerate}
        \item for each $g\in G$, there is a subset $W_g$
of $P$ with
        \begin{equation*}
         \prod_{g\in G}W_g\subset X  \text{ and } \pi_{G\backslash J}(x_0)\in \pi_{G\backslash J}\Big(\prod_{g\in G}W_g\Big);
        \end{equation*}
        \item  there is a positive number $c>0$ (independent of the choice of $g\in J$) such that for each $g\in J$ there is a continuous map $f_g:P\to W_g$ with
        \begin{equation*}
            d(x,y)\leq c\cdot d\big(f_g(x),f_g(y)\big) \text{ for all }x,y\in P.
        \end{equation*}
    \end{enumerate}
    Then
    \begin{equation*}
        mdim(X,G)\geq\delta(J)\cdot dim(P).
    \end{equation*}
\end{proposition}

\begin{proof}
    Given $\eta\geq 0$ with $\delta(J)-\eta\geq0$, there is a F$\phi$lner sequence $\{F_n\}_{n\in \mathbb{N}}$ of $G$ such that
    \begin{equation*}
        \limsup_{n\to \infty}\displaystyle\frac{|J\cap F_n|}{|F_n|}\geq\delta(J_n)-\eta.
    \end{equation*}

    Let $\{\alpha_g\}_{g\in G}$ be a sequence of positive number with $\alpha_{e_G}=1$ and $\underset{g\in G}{\sum}\alpha_g<\infty$, define a metric $\widetilde{d}$ on $P^G$ by
    \begin{equation*}
        \widetilde{d}(x,y)=\sum_{g\in G}\alpha_g d(x_g,y_g), \text{ for all } x,y\in P^G.
    \end{equation*}
     Recall that the metric $\widetilde{d}_F$ on $P^G$ with respect to $F\in \mathcal{F}(G)$ is given by
    \begin{equation*}
        \widetilde{d}_F(x,y)=\max_{g\in F}\widetilde{d}(gx,gy).
    \end{equation*}

     Define a metric $d_F$ on $P^F$ by
    \begin{equation*}
        d_{F}(u,v)=\max_{g\in F}d(u_g,v_g),\text{ for all } u,v\in P^F.
    \end{equation*}

   Given $n\in \mathbb{N}$, put $J_n=J\cap F_n$ and define a continuous map $\psi_n:P^{J_n}\to X$ by

    \begin{equation*} \psi_n(u)=\left\{
\begin{aligned}
& f_g(u_g) &   & {g\in J_n,} \\
& \pi_g(x_0) &   & {g\in G\backslash J_n,}\\
\end{aligned}
\right.
\end{equation*}
where $u=(u_g)\in P^{J_n}$. From (2), we have
\begin{equation*}
    d_{J_n}(u,v)\leq c\cdot\widetilde{d}_{J_n}(\psi_n(u),\psi_n(v))\leq c\cdot\widetilde{d}_{F_n}(\psi_n(u),\psi_n(v))
\end{equation*}
for all $u,v\in P^{J_n}$. By Lemma \ref{prop6}, we obtain
\begin{equation}\label{e2.1}
    dim_{c\epsilon}(P_{J_n},d_{J_n})\leq dim_\epsilon(X,\widetilde{d}_{F_n}).\tag{2.1}
\end{equation}
Then by Lemma \ref{prop7}, there exists a positive number $\epsilon_0>0$ such that for all $\epsilon<\epsilon_0$, we have
\begin{equation}\label{e2.2}
    dim_{c\epsilon}(P^{J_n},d_{J_n})=|J_n|\cdot dim(P).
\end{equation}

Let $\epsilon$ be small enough, from \eqref{e2.1} and \eqref{e2.2} above, we deduce that
\begin{equation*}
    mdim_\epsilon(X,\widetilde{d},G)=\lim_{n\to \infty}\displaystyle\frac{dim_\epsilon(X,\widetilde{d}_{F_n})}{|F_n|}\geq \limsup_{n\to \infty}\displaystyle\frac{|J_n|}{|F_n|}\cdot dim(P).
\end{equation*}
Hence
\begin{equation*}
    mdim(X,G)=\lim_{\epsilon\to 0}mdim_\epsilon(X,\widetilde{d},G)\geq (\delta(J_n)-\eta)\cdot dim(P).
\end{equation*}
The proof is completed by taking $\eta$ arbitrarily small.
\end{proof}

From Proposition \ref{prop10} we can immediately get the following result:

\begin{corollary}\label{cor11}
Let $G$ be a countable amenable group. Let $P$ be a polyhedron with metric $d$ and $X$ a subshift of $P^G$. Suppose that there is an increasing sequence $\{J_n\}_{n\in \mathbb{N}}$ of subsets of $G$, a sequence of positive number $\{c_n\}$, and a point $x_0\in X$ satisfying the following conditions:
\begin{enumerate}
    \item  the upper density of $\{J_n\}_{n\in \mathbb{N}}$ tends to $1$, that is,
    \begin{equation*}
        \lim_{n\to \infty}\delta(J_n)=1;
    \end{equation*}
    \item  fix $n\in \mathbb{N}$, for each $g\in G$, there is a subset $W^n_g$ of $P$ with
    \begin{equation*}
     \prod_{g\in G}W^n_g\subset X \text{ and }   \pi_{G\backslash J_n}(x_0)\in \pi_{G\backslash J_n}\Big(\prod_{g\in G}W^n_g\Big);
    \end{equation*}
    \item fix $n\in \mathbb{N}$, for each  $g\in J_n$, there is a continuous map $f^n_g:P\to W^n_g$ with
    \begin{equation*}
        d(x,y)\leq c_n\cdot d(f^n_g(x),f^n_g(y)),
    \end{equation*}
    for all $x, y\in P$, where $c_n$ is independent of the choice of $g\in J_n$.
\end{enumerate}
Then
    \begin{equation*}
        mdim(X,G)=mdim(P)=dim(P).
    \end{equation*}
\end{corollary}

\section{Proof of Theorem 1.2}
\subsection{Standing notations}
Let $(P^G,G)$ be the $G$-shift on $P$. For each subset $F$ of $G$ denote $\pi_F:P^G\to P^F$ the natural projection. Given $F\in \mathcal{F}(G)$, we say that $B_F$ is a \textbf{block} with respect to $F$ if $B_F\subset P^F$, and a block $B_{Fg}$ is said to be equal to $B_F$ \textit{up to translation} (or $B_{Fg}=B_F$ up to translation) if $\pi_F\left(g^{-1}\pi^{-1}(B_{Fg})\right)=B_F$.

Since the proof when $P$ is an $n$-dimensional cube is similar to the proof that $P=[0,1]$, for convenience, we assume $P=[0,1]$.
\subsection{Construction of $(X,G)$}
The construction will be performed by induction. In each step $k$, we shall find an irreducible tiling $\mathcal{T}_k$ with shape $\mathcal{S}_k$ and build a subset $X_k$ of $P^G$ which is generated from the family of blocks $\left\{B_{S_{k,j}}:S_{k,j}\in \mathcal{S}_k \text{ and } B_{S_{k,j}}\subset P^{S_{k,j}}\right\}$.

The construction of block $B_{S_{k,j}}$ follows the following process: In step $k-1$, assume that the irreducible finite tiling $\mathcal{T}_{k-1}$ with shape $\mathcal{S}_{k-1}$ has been found and the blocks $\left\{B_{S_{k-1, i}}:S_{k-1,i}\in \mathcal{S}_{k-1}\text{ and }B_{S_{k-1, i}}\subset P^{B_{S_{k-1, i}}}\right\}$ have been constructed.

Let $l_k$ be a positive integer we need (see its definition below). By Lemma \ref{teo4}, we can find an irreducible finite tiling $\mathcal{T}_k$ with shape $\mathcal{S}_k$ such that each $S_{k,j}\in \mathcal{S}_k$ contains at least $l_k$ many $\mathcal{T}_{k-1}$-tiles with the same shape $S_{k-1,i}$. In the absence of ambiguity, each $S_{k,j}\in \mathcal{S}_k$ has the form
\begin{equation}\label{e3.1}
     S_{k,j}=\underset{S_{k-1,i}\in\mathcal{S}_{k-1}}{\bigsqcup}S_{k-1,i}C_j(S_{k-1,i}),
\end{equation}
where \[C_j(S_{k-1,i})=\left\{c\in C(S_{k-1,i}):S_{k-1,i}c \text{ tiles }S_{k,j}\text{ up to right translation}\right\},\] and the disjoint union is also under the sense of suitable right translation.
Then for each $S_{k-1,i}\in \mathcal{S}_k$, we generated a finite mutually different family of blocks $B^1_{S_{k-1,i}},...,B^{s_{k,i}}_{S_{k-1,i}}\subset B_{S_{k-1,i}}$ ($s_{k,i}$ is a positive integer depending on $k$ and $S_{k-1,i}$). Finally, We construct block $B_{S_{k,j}}$ by filling $S_{k,j}c$ in \eqref{e3.1} ($c\in C_j(S_{k-1,i})$) with blocks $B_{k-1, i}$ and $\big\{B^l_{S_{k-1,i}}:l=1,...,s_{k, i}\big\}$ in different proportions we need. Namely, $B_{S_{k,j}}$
has the form
\begin{equation*}
     B_{S_{k,j}}=\prod_{S_{k-1,i}\in \mathcal{S}_{k-1}}\Big(\prod_{c\in C_j(S_{k-1,i})}B_{S_{k-1,i}c}\Big),
\end{equation*}
where $B_{S_{k-1,i}c}$ is equal to $B_{k-1, i}$ or one of $\big\{B^l_{S_{k-1,i}}:l=1,...,s_{k, i}\big\}$ up to translation.

 At the end of step $k$, a subset $X_k\subset P^G$ is constructed by filling each tile $S_{k,j}c_j\in \mathcal{T}_k$ with block $B_{k,j}$. That is,
 \[X_k|_{S_{k,j}c_j}=B_{S_{k,j}}\text{ up to translation}.\]
 
After infinite steps, we prove that each point $x\in \bigcap_{k\in \mathbb{N}}X_k$ is an almost periodic point. Moreover, given any $x\in \bigcap_{k\in \mathbb{N}}X_k$ we shall verify that $(X, G)=\big(\overline{\mathcal{O}(x)},G\big).$ has full mean topological dimension.

\textbf{Step -$\infty$}    In the beginning, we take two sequences \footnote{The authors in \cite{JQ} mentioned similar sequences without detailed example}
\begin{equation*}
    \left\{\eta(n,k):=1-\displaystyle\frac{1}{n+2}+\displaystyle\frac{1}{k+3}:n,k\in\mathbb{N}\text{ and }0\leq n<k\right\},
\end{equation*}
and

\begin{equation*}
    \left\{\eta(n):=1-\displaystyle\frac{1}{n+2},i.e.,\,\eta(n)=\lim_{k\to \infty}\eta(n,k)\right\}.
\end{equation*}

\begin{remark}\label{rmk3.1}
  It can be verified the two sequence $\left\{\eta(n,k):n\in \mathbb{N},k\in \mathbb{N}\right\}$ and $\left\{\eta(n):n\in \mathbb{N}\right\}$ satisfy
  \begin{enumerate}
      \item  $0\leq\eta(n,k)<1, 0<\eta(n)<1$ for all $n,k\in \mathbb{N}$;
      \item  fix $n\in \mathbb{N}$, the one-parameter sequence $\{\eta(n,k)\}_{k=n+1}^{\infty}$ is strictly decreasing, that is,
      \begin{equation*}
          \eta(n,k+1)<\eta(n,k), \text{ for all } 0\leq n<k;
      \end{equation*}
      \item  for any integer $n\in \mathbb{N}$ the one-parameter sequence $\{\eta(n,k)\}_{k=n+1}^{\infty}$ is bounded by $\eta(n)$ from below, i.e.,
      \begin{equation*}
          \eta(n,k)>\eta(n) \text{ for all } 0\leq s\leq k;
      \end{equation*}
      \item $\lim_{n\to \infty}\eta(n)=1.$
  \end{enumerate}

\end{remark}

\textbf{Step 0} Let $\mathcal{T}_0=\{e_G,g_1,...\}$ with shape $\mathcal{S}_0=\{e_G\}$. We define
\begin{equation*}
    B_{e_G}=P \text{ and } X_0=P^G.
\end{equation*}

\textbf{Step 1} Divide $P$ equally into $2$ closed intervals $P^1_1=[0,1/2]$, $P^1_2=[1/2,1]$ of length $1/2$ \big(if $P=[0,1]^d$ for $d>1$, we divide $P$ equally into $2^d$ many cubes\big). Take a positive number $l_1>2$ with
\begin{equation}\label{e3.2}
    \displaystyle\frac{l_1-2}{l_1}\geq \eta(0,1).
\end{equation}
By Theorem \ref{teo4}, we can find an irreducible finite tiling $\mathcal{T}_1$ with shape $\mathcal{S}_1$ such that each $S_{1,i}\in \mathcal{S}_1$ contains at least $l_1$ many $\mathcal{T}_1$-tiles with shape $S_0=\{e_G\}$.

For each $S_{1,i}\in\mathcal{S}_1$, select a subset $C_{i,0}=\{c_{i,0,1},c_{i,0,2}\}$ of $C(S_0)$ with $S_0C_{i,0}\subset S_{1,i}$ up to translation \big($S_0C_{i,0}\subset S_{1,i}$ up to translation means that if $T_i$ is a $\mathcal{T}_1$-tile with shape $S_{1,i}$ then $T_i$ contains $|C_{i,0}|$ $\mathcal{T}_0$-tiles with shape $S_0$\big). For each $S_{1,i}\in \mathcal{S}_1$, we define a block $B_{S_{1,i}}$ by
\begin{equation*}
    B_{S_{1,i}}=\big(\prod_{g\in S_{1,i}\backslash S_0C_{i,0}}P\big)\times \big(P^1_1\big)^{S_0c_{i,0,1}}\times \big(P^1_2\big)^{S_0c_{i,0,2}}.
\end{equation*}
Due to \eqref{e3.2}, the proportion of subintervals in $B_{S_{1,i}}$ with diameter more than $1$ is greater $\eta(0,1)$.

For each $c_{1,i}\in C(S_{1,i})$, we put
\begin{equation*}
    B_{S_{1,i}c_{{1,i}}}=B_{S_{1,i}} \text{ up to translation.}
\end{equation*}
 We produce a subset $X_1$ of $P^G$ by
\begin{equation*}
    X_1=\big\{x\in P^G:\pi_{S_{1,i}c_{1,i}}(x)\in B_{S_{1,i}c_{1,i}} \text{ for } S_{1,i}\in \mathcal{S}_1\text{ and }c_{1,i}\in C(S_{1,i})\big\}.
\end{equation*}

\textbf{Step k} Suppose the irreducible finite tiling $\mathcal{T}_{k-1}$ with shape $\mathcal{S}_{k-1}$ has been found,
the block $B_{S_{k-1,i}}$ has been generated, and the subset $X_{k-1}$ of $P^G$ has been constructed in \textbf{Step k-1}.
From the construction of $B_{S_{k-1,i}}$ we know that $B_{S_{k-1,i}}$ has the form
\begin{equation*}
    B_{S_{k-1,i}}=\prod_{g\in S_{k-1,i}}P_g^{k-1,i},
\end{equation*}
where $P_g^{k-1, i}$ is a closed interval. For each $g\in S_{k-1,i}$, we divide $P^{k-1,i}_g$ equally into $2^k$ subintervals and denote these subintervals as $P^{k-1,i}_{g,j}$, that is, if $P^{k-1,i}_g=[a,b]$ for some $a,b\in [0,1]$ with $a<b$, then
\begin{equation*}
    P^{k-1,i}_{g,j}=\Big[a+\lambda_j\displaystyle\frac{b-a}{2^k},a+(\lambda_j+1)\displaystyle\frac{b-1}{2^k}\Big],
\end{equation*}
where $\lambda_j\in \{0,...,2^k-1\}$. Notice that all $P^{k-1,i}_{g,j}$ have diameter $\displaystyle\frac{diam(P^{k-1,i}_g)}{2^k}$, where the $diam$ is in the sense of Euclidean metric.

Let $L_{k-1,i}=|S_{k-1,i}|$, where $|\cdot|$ is the cardinal operation of sets. Take a positive integer $r_{k-1}$ sufficiently large such that for each $S_{k-1,i}\in\mathcal{S}_{k-1}$ and $0\leq n\leq k$, we have

\begin{equation}\label{e3.3}
    \displaystyle\frac{\left|\big\{g\in S_{k-1,i}:diam(P^{k-1,i}_g)\geq 1/2^{n(n+1)/2}\big\}\right|\cdot r_{k-1}}{L_{k-1,i}\cdot(r_{k-1}+2^{kL_{k-1,i}})}\geq \eta(n,k),
\end{equation}

\begin{remark}\label{rem3.1}
    This inequality will ensure that the proportion of subinterval $P_g$ in $B_{k,j}$ with diameter greater than $1/2^{n(n+1)/2}$ is more than $\eta(n,k)$, that is
    \begin{equation*}
        \displaystyle\frac{\Big|\big\{g\in S_{k,j}:diam(P_g)\geq 1/2^{n(n+1)/2}, \text{ where }B_{S_{k,j}}=\prod_{g\in S_{k,j}} P_g\big\}\Big|}{|S_{k,j}|}\geq\eta(n,k)
    \end{equation*}
    for each $0\leq n\leq k$, and $B_{S_{k,j}}$ will be constructed as follows.
\end{remark}

Notice that if every $P^{k-1,i}_g$ is divided equally into $2^k$ subintervals for each $g\in S_{k-1,i}$, we shall generate as many as $2^{kL_{k-1,i}}$ subintervals for $S_{k-1,i}$. We denote them as $\big\{P^{k-1,i}_{g,j(g)}:g\in S_{k-1,i}, j(g)\in\{1,...,2^k\}\big\}$. Then we can define $2^{kL_{k-1,i}}$ mutualy different blocks $B^m_{S_{k-1,i}}$ by
\begin{equation*}
    B^m_{S_{k-1,i}}=\prod_{g\in S_{k-1,i}}P^{k-1,i}_{g,j(g)},\,m=1,...,2^{kL_{k-1,i}},
\end{equation*}
  namely, $B^m_{S_{k-1,i}}$ are  different combination of these subintervals and if $m_1\ne m$, we have $B^m_{S_{k-1,i}}\ne B^{m_1}_{S_{k-1,i}}$.

Let $l_k=\max\left\{r_{k-1}+2^{kL_{k-1,i}}: L_{k-1,i}=|S_{k-1,i}|\right\}$. By Lemma \ref{prop4} and Lemma \ref{teo4}, we can find a finite irreducible tiling $\mathcal{T}_k$ with shape $\mathcal{S}_k$, which is primely congruent to $\mathcal{T}_{k-1}$, such that each $S_{k,j}\in\mathcal{S}_k$ contains at least $l_k$ many $\mathcal{T}_{k-1}$-tiles with shape $S_{k-1,i}$ up to translation for each $S_{k-1,i}\in\mathcal{S}_{k-1}$.
 In the absence of ambiguity, we denote $C_j(S_{k-1,i})$ the subset of $C(S_{k-1,i})$ such that $S_{k, j}$ is tiled partly by  $S_{k-1,i}C_j(S_{k-1,i})$, namely, we have
\begin{equation}\label{e3.4}
    S_{k,j}=\underset{S_{k-1,i}\in\mathcal{S}_{k-1}}{\bigsqcup}S_{k-1,i}C_j(S_{k-1,i})\text{ up to  translation.}
\end{equation}
For each $C_j(S_{k-1,i})$, we let $C'_j(S_{k-1,i})$ be a subset of $C_j(S_{k-1,i})$ with $2^{kL_{k-1,i}}$ elements. Let $\varphi:C'_j(S_{k-1,i})\to \{1,...,2^{kL_{k-1,i}}\}$ be a bijective map. For each $c'\in C'_j(S_{k-1,i})$, we define a block $B_{S_{k-1,i}c'}$  as
\begin{equation*}
    B_{S_{k-1,i}c'}=B^{\varphi(c')}_{S_{k-1,i}} \text{ up to translation.}
\end{equation*}
For $c\in C_j(S_{k-1,i})\backslash C'_j(S_{k-1,i})$, we put
\begin{equation*}
    B_{S_{k-1,i}c}=B_{S_{k-1,i}} \text{ up to translation.}
\end{equation*}
Then we can get a block $B_{S_{k-1,i}C_j(S_{k-1,i})}$ as
\begin{equation*}
    B_{S_{k-1,i}C_j(S_{k-1,i})}=\Big(\prod_{c'\in C'_j(S_{k-1,i})}B_{S_{k-1,i}c'}\Big)\times\Big(\prod_{ c\in C_j(S_{k-1,i})\backslash C'_j(S_{k-1,i})}B_{S_{k-1,i}}\Big).
\end{equation*}
From \eqref{e3.4}, we define a block $B_{S_{k,j}}$ by
\begin{equation*}
    B_{S_{k,j}}=\prod_{S_{k-1,i}\in \mathcal{S}_{k-1}}B_{S_{k-1,i}C_j(S_{k-1,i})}.
\end{equation*}
Since each $B_{S_{k-1,i}}$ is a product of some closed subintervals $P^{k,j}_g$ of $P$, $B_{S_{k,j}}$ has the form
\begin{equation*}
    B_{S_{k,j}}=\prod_{g\in S_{k,j}}P^{k,j}_g,
\end{equation*}

\begin{remark}\label{rmk5}
    From the construction of $B_{k,j}$ and the inequality \eqref{e3.3} and Remark \ref{rem3.1}, the proportion of subintervals $P^{k,j}_g$ in $B_{S_{k,j}}$ with diameter over $1/2^{n(n+1)/2}$ is greater than $\eta(n,k)$ for each $0\leq n<k$.
\end{remark}

Finally, for each shape $S_{k,j}\in\mathcal{S}_k$ and $c_{k,j}\in C(S_{k,j})$, we define a block $B_{S_{k,j}c_{k,j}}$ by
\begin{equation*}
    B_{S_{k,j}c_{k,j}}=B_{S_{k,j}} \text{ up to translation,}
\end{equation*}
and generate a subset $X_k$ of $P^G$ by
\begin{equation*}
    X_k=\big\{x\in P^G:\pi_{S_{k,j}c_{k,j}}(x)\in B_{S_{k,j}c_{k,j}}, S_{k,j}\in \mathcal{S}_k \text{ and }c_{k,j}\in C(S_{k,j})\big\}.
\end{equation*}

\textbf{Step $\infty$}
After repeating the above process infinitely, we generate a sequence of irreducible finite tilings $\mathcal{T}_k$ with shape $\mathcal{S}_k$ and a sequence of subsets $X_k$ of $P^G$. Actually, for each $k\in\mathbb{N}$, $X_k$ has the form
\begin{equation*}
    X_k=\prod_{g\in G}P^k_g,
\end{equation*}
where $P^k_g$ is a subinterval of $P$, and  $X_0\supset X_1\supset\cdots\supset X_k\supset \cdots$ is a decreasing sequence. 

Put $\widetilde{X}=\bigcap_{k\in\mathbb{N}}X_k$, it is clear that $\widetilde{X}$ has the form
\begin{equation*}
    \widetilde{X}=\prod_{g\in G}P_g,
\end{equation*}
where $P_g$ is either a subinterval or a point of $P$.
\begin{remark}\label{rmk6}
    Remark \ref{rmk5} ensures that $P_g$ is non-empty and not all $P_g$ are single points. Moreover, due to the choice of the parameter $\eta(n,k)$ and $\eta(n)$, the proportion of subintervals $P_g$ in $\widetilde{X}$ with diameter more than $1/2^{n(n+1)/2}$ is greater than $\eta(n)$.
\end{remark}

In the end, take a point $x\in \widetilde{X}$, let $X=\overline{\mathcal{O}(x)}$, and the subshift is given by $(X,G)$.

\subsection{Minimality of $(X,G)$}Let $\widetilde{X}$ and $(X,G)$ be the subset of $P^G$ constructed above, we prove that
\begin{lemma}\label{lem13}
   $(X,G)$ is a minimal subshift of $(P^G,G)$ containing $\widetilde{X}$.

\end{lemma}
\begin{proof}
   \rm{(1)} We first prove $(X,G)$ is minimal. Recall that we generated a sequence of finite irreducible tilings $\mathcal{T}_k$ such that
    \begin{equation*}
            e_G\in S_{1,1}\subset S_{2,1}\subset...\subset S_{n,1}\subset...,\text{ and }\bigcup^{\infty}_{n=1}S_{n,1}=G,
        \end{equation*}
        where $\mathcal{S}_{k}=\{S_{k,1},...,S_{k,m_k}\}$ is the shape of $\mathcal{T}_k$.
        Moreover, $\{S_{k,1}\}_{k=1}^{+\infty}$ is a F$\phi$lner sequence of $G$ and the center of each shape is syndetic.

        Let $d$ be the Euclidean metric on $P$, for each $S_{k,1}\in\mathcal{S}_k$, we define
        \begin{equation*}
            d_{S_{k,1}}(x,y)=\max_{g\in S_{k,1}}d(x_g,y_g),
        \end{equation*}
        where $x=(x_g),y=(y_g)\in P^G$.

        For each $n,k\in \mathbb{N}$, we write
        \begin{equation*}
            \big[S_{k,1},n\big]=\big\{y\in P^G:d_{S_{k,1}}(x,y)<1/n\big\},
        \end{equation*}
        $[S_{k,1},n]$ forms a neighborhood base of $x$.
        Now it suffices to show that $N_G(x,[S_{k,1},n])$ is syndetic for each $k,n\in \mathbb{N}$. 
        
        Recall that the block $B_{k,1}$ has the form
        \begin{equation*}
            B_{S_{k,1}}=\prod_{g\in S_{k,1}}P^{k,1}_g,
        \end{equation*}
        where $P^{k,1}_g$ is a subinterval of $[0,1]$.
        Let $m>k$, start from \textbf{Step k}, after $m-k$ steps, each $P^{k,1}_g$ will be divided equally into $2^{(m+1)(m-2k)/2}$ subintervals $P^{k,1,w}_g$ with $diam(P^{k,1,w}_g)=diam(P^{k,1}_g)/2^{(m+1)(m-2k)/2}$, where $w=1,2,...,2^{(m+1)(m-2k)/2}$. By combining these subintervals to form new blocks $B^m_{S_{k,1}}$, there is a block $B^m_{S_{k,1}}$ such that $\pi_{S_{k,1}}(x)\in B^m_{S_{k,1}}$.
        Since $\{\mathcal{T}_k\}$ is a primely congruent sequence and due to the construction of $B_{m,1}$, there is a $g_m\in G$ and $c_k\in C(S_{k,1})$ with $S_{k,1}c_kg_m\subset S_{m,1}$. Therefore, we can obtain
        \begin{equation*}
            \pi_{S_{k,1}}(x)\in \pi_{S_{k,1}}(B_{S_{m,1}g_m^{-1}c^{-1}_{k}})=B^m_{{S_{k,1}}},
        \end{equation*}
       and $B_{S_{m,1}g^{-1}c^{-1}_{k}}=B_{S_{m,1}}$ up to translation, which means

        \begin{equation*}
            d_{S_{k,1}}(c_kg_mx,x)<\displaystyle\frac{diam\big(P^{k,1}_g\big)}{2^{(m+1)(m-2k)/2}}.
        \end{equation*}
      
      We can require $diam(P^{k,1}_g)/2^{(m+1)(m-2k)/2}<1/n$ by letting $m$ big enough, , then
        \begin{equation*}
        d_{S_{k,1}}(c_kg_mx,x)<1/n,
        \end{equation*}
        that is
        \begin{equation*}
            c_kg_mx\in \big[S_{k,1},n\big].
        \end{equation*}

        Recall that
        \begin{equation*}
            X_m=\prod_{g\in G}P^m_g=\Big(\prod_{c_m\in C(S_{m,1})}B_{S_{m,1}c_m}\Big)\times \Big(\prod_{G\backslash S_{m,1}C(S_{m,1})}P^m_g\Big),
        \end{equation*}
        where $B_{S_{m,1}c_m}=B_{m,1}$ up to translation, then for each $c_m\in C(S_{m,1})$
        we obtain
        \begin{equation*}
            C(S_{m,1})c_kg_m\subset N_G(x,[S_{k,1},n]).
        \end{equation*}
        Since $C(S_{m,1})c_kg_m$ is syndetic, we conclude that $(X,G)$ is a minimal subshift.

        \rm{(2)} Next, we prove that $X$ contains $\widetilde{X}$. Let $y\in \widetilde{X}$, we shall prove that for each $k,n\in\mathbb{N}$, there is a $g\in G$ such that
        \begin{equation*}
            d_{S_{k,1}}(gx,y)<1/n.
        \end{equation*}
        Define
        \begin{equation*}
            \big[y,S_{k,1},n\big]=\big\{z\in P^G:d_{S_{k,1}}(y,z)<1/n\big\},
        \end{equation*}
        $\big[y,S_{k,1},n\big]$ forms a neighborhood base of $y$. Recall that the block $B_{S_{k,1}}$ has the form
        \begin{equation*}
            B_{S_{k,1}}=\prod_{g\in S_{k,1}}P^{k,1}_g,
        \end{equation*}
        where $P^{k,1}_g$ is a subinterval of $P$.

        Let $m>k$, from the construction of $B_{S_{k,1}}$ at \textbf{Step k}, after $m-k$ steps, each $P^{k,1}_g$ will be divided equally into $2^{(m+1)(m-2k)/2}$ subintervals $P^{k,l,w}_g$ with $diam(P^{k,l,w}_g)=diam(P^{k,l}_g)/2^{(m+1)(m-2k)/2}$, where $w=1,2,...,2^{(m+1)(m-2k)/2}$, and the block $B_{m,1}$ is constituted by $\prod_{g\in S_{k,1}}P^{k,1,w(g)}_g$, moreover, $\pi_{S_{k,1}}(y)$ belongs to one of the $\prod_{g\in S_{k,1}}P^{k,1,w(g)}_g$. Therefore, due to the construction of $\widetilde{X}$ and the choice of $x$, we can find an $h\in G$ such that
        \begin{equation*}
            \pi_{S_{k,1}}(hx), \pi_{S_{k,1}}(y)\in \prod_{g\in S_{k,1}}P^{k,1,w(g)}_g.
        \end{equation*}
       Let $m\in\mathbb{N}$ be big enough such that
       \begin{equation*}
           \displaystyle\frac{3\cdot diam(P^{k,l}_g)}{2^{(m+1)(m-2k)/2}}<1/n,
       \end{equation*}
       then
       \begin{equation*}
           d_{S_{k,1}}(hx,y)<1/n,
       \end{equation*}
       for some $h\in G$.
       Therefore, $hx\in \big[y,S_{k,1},n\big]$, which completes the proof of Lemma \ref{lem13}.
    \end{proof}

    \subsection{$(X,G)$ has full mean topological dimension}
In the end, we shall prove that the minimal subshift $(X, G)$ has full mean topological dimension.
    \begin{lemma}\label{lem14}
        Let $(X,G)$ be the subshift constructed above, then $mdim(X,G)=dim(P)$.
    \end{lemma}
    \begin{proof}
        Recall that the subset $\widetilde{X}$ has the form
        \begin{equation*}
            \widetilde{X}=\prod_{g\in G}P_g,
        \end{equation*}
        where $P_g$ is a subinterval of $P$. From Remark \ref{rmk6} we know that the proportion of subintervals $P_g$ with diameter greater than $1/2^{n(n+1)/2}$ is over $\eta(n)$.

        Put
        \begin{equation*}
            J_n=\left\{g\in G:diam(P_g)\geq 1/2^{n(n+1)/2}\right\}.
        \end{equation*}
Since $\{S_{k,1}\}$ is a F$\phi$lner sequence of $G$, then
\begin{equation*}
    \delta(J_n)\geq \limsup_{k\to \infty}\displaystyle\frac{|J_n\cap S_{k,1}|}{|S_{k,1}|}>\eta(n).
\end{equation*}
From Remark \ref{rmk3.1} (4), we have $\lim_{n\to \infty}\delta(J_n)=1$.

For each $g\in G$ such that $P_g$ the subinterval with diameter greater than  $1/2^{n(n+1)/2}$. We can define a continuous linear map $f^n_g$ from $P$ to $P_g$ satisfying
\begin{equation*}
    d(x,y)\leq 2^{n(n+1)/2}d(f^n_g(x),f^n_g(y)).
\end{equation*}
Take $x\in X$, it is clear that $(X,G)$ satisfies all the conditions of Corollary \ref{cor11}. Hence $(X,G)$ has full mean topological dimension.
    \end{proof}

    \begin{theorem}[= Theorem 2]
        Let $G$ be an infinite countable amenable group and $P$ a polyhedron with topological dimension $dim(P)$, there is a minimal subshift $X$ of $P^G$ such that its mean topological dimension $mdim(X, G)=dim(P)$.
    \end{theorem}
    \begin{proof}
        It directly follows from Lemma \ref{lem13} and Lemma \ref{lem14}.
    \end{proof}

\section*{Acknowledgement}
The second author was supported by NNSF of China (Grant No. 12201120).

\vspace{0.5cm}

\address{Department of Mathematics, Nanjing University, Nanjing 210093, People's Republic of China}

\textit{E-mail}: \texttt{yzy199707@gmail.com}

\address{School of Mathematics and Statistics, Fuzhou University, Fuzhou 350116, People's Republic of China}

\textit{E-mail}: \texttt{xzb2020@fzu.edu.cn}

\end{document}